\documentclass[letterpaper, 11pt]{article}
\usepackage{fullpage}

\usepackage{epsfig,graphicx,graphics,color}
\usepackage{amsmath,amsthm,amssymb,color,comment,url}
\usepackage{mathtools,thmtools}
\usepackage{thm-restate}
\usepackage{caption}
\usepackage[position=b]{subcaption}
\usepackage[unicode]{hyperref}
\hypersetup{
    colorlinks=true,       
    linkcolor=blue,        
    citecolor=red,         
    filecolor=magenta,     
    urlcolor=cyan,         
    linktocpage=true
}
\usepackage[sort,nocompress]{cite}
\usepackage{titling} 
\thanksmarkseries{alph}
\usepackage{enumerate}

\def\final{0}  
\def\iflong{\iffalse}
\ifnum\final=0  
\usepackage[dvipsnames]{xcolor}
\newcommand{\yu}[1]{{\color{red}[{\tiny \textbf{Yu:} \bf #1}]\marginpar{\color{red}*}}}
\newcommand{\yutaro}[1]{{\color{red}[{\tiny \textbf{Yutaro:} \bf #1}]\marginpar{\color{red}*}}}
\newcommand{\tamas}[1]{{\color{red}[{\tiny \textbf{Tam\'as:} \bf #1}]\marginpar{\color{red}*}}}
\newcommand{\kristof}[1]{{\color{red}[{\tiny \textbf{Krist\'of:} \bf #1}]\marginpar{\color{red}*}}}
\else 
\newcommand{\yu}[1]{}
\newcommand{\yutaro}[1]{}
\newcommand{\tamas}[1]{}
\newcommand{\kristof}[1]{}
\fi

\numberwithin{equation}{section}

\theoremstyle{plain}
\newtheorem{theorem}{Theorem}[section]
\newtheorem{lemma}[theorem]{Lemma}

\newtheorem{claim}[theorem]{Claim}

\theoremstyle{definition}
\newtheorem{definition}[theorem]{Definition}

\theoremstyle{remark}
\newtheorem{remark}{Remark}

\newcommand{\RR}{{\mathbb{R}}}
\newcommand{\ZZ}{{\mathbb{Z}}}
\newcommand{\opt}{{\mathrm{opt}}}
\newcommand{\lopt}{{\mathrm{lexopt}}}
\newcommand{\bM}{{\mathbf{M}}}
\newcommand{\cI}{{\mathcal{I}}}
\newcommand{\cF}{{\mathcal{F}}}

\renewcommand{\sp}{{\mathrm{span}}}

\makeatletter
\def\namedlabel#1#2{\begingroup
    #2%
    \def\@currentlabel{#2}%
    \phantomsection\label{#1}\endgroup
}
\makeatother

\usepackage[algo2e,boxruled,linesnumbered,procnumbered]{algorithm2e}
\makeatletter
\renewcommand{\algocf@caption@boxruled}{%
  \hrule
  \hbox to \hsize{%
    \vrule\hskip-0.4pt
    \vbox{   
       \vskip\interspacetitleboxruled%
       \unhbox\algocf@capbox\hfill
       \vskip\interspacetitleboxruled
       }%
     \hskip-0.4pt\vrule%
   }\nointerlineskip%
}%
\makeatother

\usepackage{etoolbox}
\newcounter{algoline}

\AtBeginEnvironment{algorithm}{\setcounter{algoline}{0}}

\title{Approximation by Lexicographically Maximal Solutions\\in Matching and Matroid Intersection Problems}
\author{
Krist\'of B\'erczi\thanks{MTA-ELTE Momentum Matroid Optimization Research Group and MTA-ELTE Egerv\'ary Research Group, Department of Operations Research, E\"otv\"os Lor\'and University, Budapest, Hungary. Email: \texttt{\{kristof.berczi,tamas.kiraly\}@ttk.elte.hu}} 
\and
Tam\'as Kir\'aly\footnotemark[1]
\and
Yutaro Yamaguchi\thanks{Department of Information and Physical Sciences, Graduate School of Information Science and Technology, Osaka University, Osaka, Japan. Email: \texttt{yutaro.yamaguchi@ist.osaka-u.ac.jp}} 
\and
Yu Yokoi\thanks{Principles of Informatics Research Division, National Institute of Informatics, Tokyo, Japan. Email: \texttt{yokoi@nii.ac.jp}}}

\date{\empty}

\begin{document}
\maketitle
\thispagestyle{empty}

\begin{abstract}
We study how good a lexicographically maximal solution is in the weighted matching and matroid intersection problems.
A solution is lexicographically maximal if it takes as many heaviest elements as possible, and subject to this, it takes as many second heaviest elements as possible, and so on.
If the distinct weight values are sufficiently dispersed, e.g., the minimum ratio of two distinct weight values is at least the ground set size, then the lexicographical maximality and the usual weighted optimality are equivalent.
We show that the threshold of the ratio for this equivalence to hold is exactly $2$.
Furthermore, we prove that if the ratio is less than $2$, say $\alpha$, then a lexicographically maximal solution achieves $(\alpha/2)$-approximation, and this bound is tight.

\bigskip

\noindent \textbf{Keywords:} Matching, Matroid intersection, Lexicographical maximality, Approximation ratio
\end{abstract}

\clearpage
\setcounter{page}{1}

\section{Introduction}
Matching in bipartite graphs is one of the fundamental topics in combinatorics and optimization.
Due to its diverse applications, various optimality criteria of matchings have been proposed based on the number of edges, the total weight of edges, etc.
The concept of \emph{rank maximality} is one of them, which is especially studied in the context of matching problems under preferences \cite{irving2006rank,michail2007reducing}.
In this setting, we are given a partition $\{E_1, E_2, \dots, E_k\}$ of the edge set $E$ that represents a priority order, and a matching $M \subseteq E$ is \emph{rank-maximal} if $|M \cap E_1|$ is maximized, and subject to this, $|M \cap E_2|$ is maximized, and so on.
As pointed out in several papers \cite{irving2006rank,michail2007reducing,kamiyama2013matroid,huang2019exact} (on generalized problems), the problem of finding a rank-maximal solution can be reduced to the usual weight-maximization setting by using sufficiently dispersed weights, e.g., by assigning $|E|^{k-i}$ for each element in $E_i$. 

In this paper, we explore the relation between rank-maximality and optimality in the weighted setting for two generalizations of bipartite matching: matching in general graphs and matroid intersection.
To avoid confusion, we hereafter replace the term ``rank'' with ``lex(icographical)'' because we also deal with ranks in matroids.
The main results of the paper (Theorems~\ref{thm:matching} and \ref{thm:matroid_intersection}) state that, in both problems, the equivalence between lex-maximality and weighted optimality holds if the minimum ratio of two distinct weight values is larger than $2$.
This implies that we can choose the base of exponential weights as any constant larger than $2$ (instead of $|E|$) in the aforementioned reduction.
Furthermore, we show that if the minimum ratio is at most $2$, say $\alpha$, then a lex-maximal solution achieves $(\alpha/2)$-approximation of the maximum weight, and this bound is tight.

\section{Results}
We assume the basic notation and terminology on graphs and matroids (see, e.g., \cite{schrijver2003combinatorial}).

Throughout the paper, let $E$ be a finite ground set and $w \colon E \to \RR_{> 0}$ be a positive weight function on $E$.
For a subset $X \subseteq E$, the weight $w(X)$ is defined as the sum of its elements.
Let $w_1, w_2, \dots, w_k$ be the distinct values of $w$ in descending order.
We assume $k \ge 2$ (otherwise, the results are trivial), and define
\[\alpha_w \coloneqq \min_{1 \le i \le k-1} \frac{w_{i}}{w_{i+1}}.\]
For a subset $X \subseteq E$ and $1 \leq i \leq k$, we denote by $X_i = \{\, e \in X \mid w(e) = w_i \,\}$ the restriction of $X$ to the elements of weight $w_i$, and define $X_{\le i} \coloneqq \bigcup_{j \le i} X_j$ and accordingly $X_{\ge i}$, $X_{< i}$, and $X_{> i}$.
For two subsets $X, Y \subseteq E$, we say that $X$ is \emph{lex-larger} than $Y$ (or $Y$ is \emph{lex-smaller} than $X$) if there exists an index $i$ such that $|X_j| = |Y_j|$ for any $j < i$ and $|X_i| > |Y_i|$.
For a family $\cF \subseteq 2^E$ of subsets of $E$, we say that $X \in \cF$ is \emph{lex-maximal} if no $Y \in \cF$ is lex-larger than $X$.
Note that a lex-maximal subset $X \in \cF$ may not be unique but the sequence $|X_1|, |X_2|, \dots, |X_k|$ is unique.

A \emph{weighted matching instance} consists of an undirected graph $G = (V, E)$ and a weight function $w$ on the edge set.
We denote by $\opt(G, w)$ the optimal value, i.e., the maximum weight of a matching in $G$.
In addition, we denote by $\lopt(G, w)$ the weight of a lex-maximal matching in $G$ (where $\cF$ is the family of matchings in $G$), which takes as many edges of weight $w_1$ as possible, and subject to this, takes as many edges of weight $w_2$ as possible, and so on.

A \emph{weighted matroid intersection instance} consists of two matroids $\bM_1 = (E, \cI_1)$ and $\bM_2 = (E, \cI_2)$ and a weight function $w$ on the common ground set.
We denote by $\opt(\bM_1, \bM_2, w)$ the optimal value, i.e., the maximum weight of a common independent set in $\cI_1 \cap \cI_2$.
In addition, we denote by $\lopt(\bM_1, \bM_2, w)$ the weight of a lex-maximal common independent set (where $\cF = \cI_1 \cap \cI_2$), which takes as many elements of weight $w_1$ as possible, and subject to this, takes as many elements of weight $w_2$ as possible, and so on.

The main theorems are stated as follows.

\begin{theorem}\label{thm:matching}
Let $(G, w)$ be a weighted matching instance.
If $\alpha_w \le 2$, then
\[\lopt(G, w) \ge \frac{\alpha_w}{2} \cdot \opt(G, w).\]
If $\alpha_w > 2$, then a lex-maximal matching is a maximum-weight matching, and vice versa.
\end{theorem}

\begin{theorem}\label{thm:matroid_intersection}
Let $(\bM_1, \bM_2, w)$ be a weighted matroid intersection instance.
If $\alpha_w \le 2$, then
\[\lopt(\bM_1, \bM_2, w) \ge \frac{\alpha_w}{2} \cdot \opt(\bM_1, \bM_2, w).\]
If $\alpha_w > 2$, then a lex-maximal common independent set is a maximum-weight common independent set, and vice versa.
\end{theorem}

We remark that the approximation ratio is tight as the weighted bipartite matching problem is included as a common special case.
Let $G = (V, E)$ be a bipartite graph with $V = \{1, 2, 3, 4\}$ and $E = \{\,e_1 = \{1, 3\},\, e_2 = \{2, 3\},\, e_3 = \{2, 4\}\,\}$.
Define $w(e_1) = 1$, $w(e_2) = x \in (1, 2]$, and $w(e_3) = 1$.
We then have $\alpha_w = x$ (as $w_1 = x$ and $w_2 = 1$), $\opt(G, w) = 2$, and $\lopt(G, w) = x$.
We also remark that, when $x = 2$, there are two maximum-weight matchings $\{e_2\}$ and $\{e_1, e_3\}$; the former is lex-maximal whereas the latter is not.
This means that ``and vice versa'' is not true when $\alpha_w = 2$ in both theorems.

\section{Proofs}
We prove both theorems using the same strategy.
The key definition and lemmas are as follows.

\begin{definition}\label{def:eligible}
Let $X \subseteq E$ be a feasible solution that is not lex-maximal in a weighted matching or matroid intersection instance, and $i$ be the smallest index such that $X_{\le i}$ is not lex-maximal in the restricted instance whose ground set is $E_{\le i}$.
We say that a feasible solution $Y \subseteq E$ in the original instance is \emph{eligible} if the following three conditions are satisfied:
\begin{itemize}
\item $|Y_j| = |X_j|$ for any $j < i$,
\item $|Y_i| = |X_i| + 1$, and
\item $|X_{> i} \setminus Y_{> i}| \le 2$.
\end{itemize}
\end{definition}

Intuitively, an eligible solution lexicographically improves the original solution at the most significant improvable class by sacrificing at most two lighter elements.

\begin{lemma}\label{lem:matching}
For any weighted matching instance $(G, w)$ and any matching $X$ that is not lex-maximal, there exists an eligible matching $Y$.
\end{lemma}

\begin{lemma}\label{lem:matroid_intersection}
For any weighted matroid intersection instance $(\bM_1, \bM_2, w)$ and any common independent set $X$ that is not lex-maximal, there exists an eligible common independent set $Y$.
\end{lemma}

We here prove Theorems~\ref{thm:matching} and \ref{thm:matroid_intersection} using Lemmas~\ref{lem:matching} and \ref{lem:matroid_intersection}, whose proofs are provided later.

\begin{proof}[Proof of Theorems~\ref{thm:matching} and \ref{thm:matroid_intersection}]
Let $X^*$ be an optimal solution.
Starting with $X = X^*$, we repeatedly update $X$ to an eligible $Y$ until it becomes lex-maximal.
Let $Y^*$ be the lex-maximal solution that is finally obtained.

In any update from $X$ to $Y$, we have
\[w(Y) = w(Y_{\le i}) + w(Y_{> i}) \ge w(X_{\le i}) + w_i + w(X_{> i}) - 2w_{i+1} \ge w(X) - \frac{2 - \alpha_w}{\alpha_w} \cdot w_i.\]
If $\alpha_w > 2$, then we have $w(Y) > w(X)$, which cannot happen at the beginning when $X = X^*$.
Thus, $X^*$ is lex-maximal.
As any lex-maximal solution has the same weight, we also see that any lex-maximal solution is optimal, and we are done with the second statements.

Suppose that $\alpha_w \le 2$.
Since we always have $|Y_j| = |X_j| = |Y_j^*|$ $(j < i)$ and $|Y_i| = |X_i| + 1 \le |Y_i^*|$, we see that $i$ is nondecreasing during the process and each $w_i$ appears on the right-hand side at most $|Y_i^*|$ times in total.
Thus, by repeating the above inequalities, we obtain
\[w(Y^*) \ge w(X^*) - \frac{2 - \alpha_w}{\alpha_w}\sum_{i = 1}^k (w_i \cdot |Y_i^*|) = w(X^*) - \frac{2 - \alpha_w}{\alpha_w} \cdot w(Y^*),\]
which implies the first statements.
\end{proof}

\subsection{Matching: Proof of Lemma~\ref{lem:matching}}
We prove Lemma~\ref{lem:matching} by contradiction.
Suppose to the contrary that there exists a counterexample, and take a minimal one in the following sense.
First, the ground set $E$ is minimized as the first priority, and subject to this, a counterexample matching $X$ (that is not lex-maximal but admits no eligible matching $Y$) and a lex-maximal matching $Z$ are taken so that the symmetric difference $X \triangle Z = (X \setminus Z) \cup (Z \setminus X)$ is minimized.

By the minimality, we have $E = X \cup Z  = X \triangle Z$.
Indeed, if there exists an edge $e \in E \setminus (X \cup Z)$, then we obtain a smaller counterexample by removing $e$ (no eligible matching can newly appear), a contradiction. Similarly, if there exists $e \in X \cap Z$, then we obtain a smaller counterexample again by removing $e$ (note that no edge is adjacent to $e$ as $X$ and $Z$ are matchings), a contradiction.

\begin{claim}\label{cl:same_matching}
There are no adjacent edges of the heaviest weight. 
\end{claim}

\begin{proof}
Suppose to the contrary that $e_1 = \{v, u_1\} \in X$ and $e_2 = \{v, u_2\} \in Z$ are adjacent (at $v$) and $w(e_1) = w(e_2) = w_1$.
If $u_1 = u_2$ (i.e., $e_1$ and $e_2$ are parallel), then $X' = (X \setminus \{e_1\}) \cup \{e_2\}$ is a smaller counterexample, a contradiction.
Otherwise, consider the instance obtained by contracting the two edges $e_1$ and $e_2$, i.e., by merging $u_1$ and $u_2$ into a single vertex, removing the vertex $v$ together with the incident edges, and restricting the weight function to the remaining set of edges.

As $X$ is a matching that is not lex-maximal, this is true for $X' = X \setminus \{e_1\}$ after the contraction.
By the minimality of the counterexample, there exists an eligible matching $Y' \subseteq E \setminus \{e_1, e_2\}$.
Since at most one of $u_1$ and $u_2$ is matched by $Y'$ in the original graph, we can add $e_1$ or $e_2$ to $Y'$ to obtain a matching $Y$ in the original instance.
Moreover, the $Y$ thus obtained is eligible by definition even if $e_2$ is added since $e_1 \in X_{\le i}$ for any $i \ge 1$, contradicting our indirect assumption.
\end{proof}

\begin{claim}\label{cl:larger_matching}
$X_1 = \emptyset$. 
\end{claim}

\begin{proof}
Suppose to the contrary that there exists an edge $e \in X_1$.
Then, by Claim~\ref{cl:same_matching}, we have $w(f) < w(e) = w_1$ for each adjacent edge $f \in Z$.
Thus, we can obtain a lex-larger matching $Z'$ from $Z$ by adding $e$ and by removing all the (at most two) adjacent edges, a contradiction.
\end{proof}

By Claim \ref{cl:larger_matching}, we have $|X_1| = 0 < |Z_1|$.
Let $f \in Z$ be an edge with $w(f) = w_1$.
Then, we can obtain an eligible matching $Y$ from $X$ by adding $f$ and by removing all the (at most two) adjacent edges, a contradiction.
This concludes the proof of the lemma.

\begin{remark}
We could give algorithmic proofs, e.g., what is similar to the matroid intersection case given in the next section, or what is based on the property of lex-maximal matchings that any lex-maximal matching in $E_{< i}$ can be augmented to some lex-maximal matching in $E_{\le i}$ along alternating paths (which is implicitly observed and utilized in the algorithm in \cite{irving2006rank}) as follows.

Let $X$ be a matching that is not lex-maximal, $Y$ be a lex-maximal matching, and $i$ be the smallest index in Definition~\ref{def:eligible}.
Since both $X_{< i}$ and $Y_{< i}$ are lex-maximal in $E_{< i}$, each connected component of $X_{< i} \triangle Y_{< i}$ is an alternating path/cycle which contains the same number of edges in $X_j$ and in $Y_j$ for every $j < i$ (otherwise, by flipping edges along it, one is made a lex-larger matching in $E_{< i}$, a contradiction).
In particular, all of such paths are of even length.

Let us consider $X_{\le i} \triangle Y_{\le i}$.
Since $Y_{\le i}$ is lex-maximal in $E_{\le i}$ but $X_{\le i}$ is not, there exists an alternating path/cycle which contains strictly more edges in $Y_i$ than edges in $X_i$.
Such an alternating path/cycle must be obtained by connecting edges in $X_i$ and in $Y_i$ alternately with even-length (possibly zero-length) alternating paths in $X_{< i} \triangle Y_{< i}$ between them;
hence, the number of edges in $Y_i$ is larger than that of edges in $X_i$ by exactly one and it cannot be a cycle.
Thus, there exists an alternating path $P$ in $X_{\le i} \triangle Y_{\le i}$ such that $|P \cap Y_i| = |P \cap X_i| + 1$ and $|P \cap Y_{j}| = |P \cap X_j|$ for every $j < i$.
By flipping the edges of $X_{\le i}$ along $P$ and removing at most two edges in $X_{> i}$ incident to the end vertices of $P$, we obtain an eligible matching.
\end{remark}

\subsection{Matroid Intersection: Proof of Lemma~\ref{lem:matroid_intersection}}
The proof relies on (the correctness of) an augmenting path algorithm for the weighted matroid intersection problem, which was first described in \cite{lawler1975matroid}.
We first review the basic facts based on \cite[Sections~41.2 and 41.3]{schrijver2003combinatorial}.

Let $(\bM_1, \bM_2, w)$ be a weighted matroid intersection instance.
For each matroid $\bM_i$ $(i \in \{1, 2\})$, we denote the independent set family, the rank function, and the span function\footnote{This is also called the closure function; we follow the terminology of \cite{schrijver2003combinatorial} for simplicity.} by $\cI_i \subseteq 2^E$, $r_i \colon 2^E \to \ZZ_{\ge 0}$, and $\sp_i \colon 2^E \to 2^E$, respectively.
For a set $X$ and elements $x \in X$ and $y \not\in X$, we write $X \setminus \{x\}$ and $X \cup \{y\}$ as $X - x$ and $X + y$, respectively.

Let $I\in\cI_1\cap \cI_2$ be a common independent set.
The \emph{exchangeability graph} with respect to $I$ is a directed bipartite graph $D = (E \setminus I, I; A)$ defined as follows.
Let $A \coloneqq A_1 \cup A_2$, where
\begin{align*}
  A_1 \coloneqq&\ \{\, (y, x) \mid x \in E \setminus I,~y \in I,~I + x - y \in \cI_1 \,\},\\
  A_2 \coloneqq&\ \{\, (x, y) \mid x \in E \setminus I,~y \in I,~I + x - y \in \cI_2 \,\}.
\end{align*}
We also define
\begin{align*}
  S &\coloneqq \{\, s \in E \setminus I \mid I + s \in \cI_1 \,\},\\
  T &\coloneqq \{\, t \in E \setminus I \mid I + t \in \cI_2 \,\},
\end{align*}
where the elements in $S$ and in $T$ are called \emph{sources} and \emph{sinks}, respectively. 
Note that $A_1$ and $S$ depend only on $\bM_1$, and $A_2$ and $T$ depend only on $\bM_2$.
Let $c \colon E \to \RR$ be a cost function on the vertex set defined as follows:
\begin{align}\label{eq:cost}
  c(e) &\coloneqq \begin{cases}
    w(e) & (e \in I),\\
    -w(e) & (e \in E \setminus I).
  \end{cases}
\end{align}
An $S$--$T$ path $P$ in $D$ is \emph{cheapest} if the total cost of its vertices is minimized.
Subject to this, $P$ is \emph{shortest} if the number of its vertices is minimized.

For a nonnegative integer $\ell$, a common independent set $I \in \cI_1 \cap \cI_2$ is said to be \emph{$\ell$-extreme} if $w(I)$ is maximized subject to $I \in \cI_1 \cap \cI_2$ and $|I| = \ell$.
The following lemma leads to a simple augmenting path algorithm for the weighted matroid intersection problem.
The key is that one can augment any $\ell$-extreme solution to some $(\ell + 1)$-extreme solution (if exists) by simply exchanging elements along a path.

\begin{lemma}[cf.~{\cite[Theorems~41.3, 41.5, and 41.6]{schrijver2003combinatorial}}]\label{lem:extreme}
Let $I \in \cI_1 \cap \cI_2$ be an $\ell$-extreme common independent set, and suppose that there exists a common independent set $J \in \cI_1 \cap \cI_2$ with $|J| > |I|$.
Then, $D$ contains an $S$--$T$ path, which may consist of a single vertex in $S \cap T$.
Let $P$ be a shortest cheapest $S$--$T$ path\footnote{Actually, $P$ does not need to be shortest but it suffices that $P$ does not admit any shortcut.} in $D$ with respect to the cost function $c$ defined as \eqref{eq:cost}.
Then, no inner vertex of $P$ is a source or a sink, and $I \triangle P$ is an $(\ell + 1)$-extreme common independent set.
\end{lemma}

Now we start the proof of Lemma~\ref{lem:matroid_intersection}.
Let $X \in \cI_1 \cap \cI_2$ be a common independent set that is not lex-maximal, and let $i$ be the smallest index such that $X_{\le i}$ is not lex-maximal in the restricted instance whose ground set is $E_{\le i}$.

\begin{claim}\label{cl:augment}
There exists a common independent set $Y' \subseteq E_{\le i}$ such that $|Y'_j| = |X_j|$ for any $j < i$, $|Y'_i| = |X_i| + 1$, $\sp_1(X_{\le i}) \subseteq \sp_1(Y')$, and $\sp_2(X_{\le i}) \subseteq \sp_2(Y')$.
\end{claim}

\begin{proof}
Define an auxiliary weight function $w' \colon E_{\le i} \to \RR_{> 0}$ by $w'(e) \coloneqq n^{i-j}$ for each $e \in E_j$ $(j = 1, 2, \dots, i)$, where $n = |E_{\le i}|$.
As remarked in the introduction, the lexicographical order coincides with the weighted order, i.e., for any two subsets $Z, Z' \subseteq E_{\le i}$, $Z$ is lex-larger than $Z'$ if and only if $w'(Z) > w'(Z')$.

Let $\ell \coloneqq |X_{\le i}|$.
Then, $X_{\le i}$ is $\ell$-extreme in the restricted instance $(\bM_1 | E_{\le i}, \bM_2 | E_{\le i}, w')$, which has a larger common independent set.
By Lemma~\ref{lem:extreme}, one can obtain an $(\ell + 1)$-extreme common independent set $Y' \subseteq E_{\le i}$ by flipping $X_{\le i}$ along a source-sink path $P$ in the exchangeability graph.
Note that $|Y'_j| = |X_j|$ for any $j < i$ and $|Y'_i| = |X_i| + 1$ by the choice of $i$ and the definition of $w'$.

If $P$ consists of a single vertex (which is a source and a sink), then the claim immediately follows since $Y' = X_{\le i} + y$ for some element $y \in E_i$.
Otherwise, let $s$ and $t$ be the first and last vertices of $P$, respectively.
By the definitions of the exchangeability graph and the sources and sinks, $y \in \sp_1(X_{\le i}) \subseteq \sp_1(X_{\le i} + s)$ for every $y \in P - s$ and $y \in \sp_2(X_{\le i}) \subseteq \sp_2(X_{\le i} + t)$ for every $y \in P - t$ (recall that any $y \in P - s$ is not a source and any $y \in P - t$ is not a sink).
Hence, we have $\sp_1(X_{\le i} \cup P) = \sp_1(X_{\le i} + s)$ and $\sp_2(X_{\le i} \cup P) = \sp_2(X_{\le i} + t)$, and then
\begin{align*}
  \ell + 1 &= |Y'| = r_1(Y') \le r_1(X_{\le i} \cup P) = r_1(X_{\le i} + s) \le r_1(X_{\le i}) + r_1(s) = |X_{\le i}| + 1 = \ell + 1,\\
  \ell + 1 &= |Y'| = r_2(Y') \le r_2(X_{\le i} \cup P) = r_2(X_{\le i} + t) \le r_2(X_{\le i}) + r_2(t) = |X_{\le i}| + 1 = \ell + 1,
\end{align*}
which implies that the equality holds everywhere.
Thus we have $\sp_1(Y') = \sp_1(X_{\le i} + s) \supseteq \sp_1(X_{\le i})$ and $\sp_2(Y') = \sp_2(X_{\le i} + t) \supseteq \sp_2(X_{\le i})$, which completes the proof.
\end{proof}

Take a subset $Y' \subseteq E_{\le i}$ satisfying the conditions in Claim~\ref{cl:augment}.
We then obtain an eligible common independent set $Y \subseteq E$ from $Y' \cup X_{> i}$ by removing at most two elements in $X_{> i}$ as follows.
If $r_1(Y' \cup X_{> i}) = r_2(Y' \cup X_{> i}) = |X| + 1$, then $Y' \cup X_{> i} \in \cI_1 \cap \cI_2$ and we do not need to remove any element.
Suppose that $r_1(Y' \cup X_{> i}) < |X| + 1$.
As $\sp_1(X_{\le i}) \subseteq \sp_1(Y')$, we have $r_1(Y' \cup X_{> i}) \ge r_1(X) = |X|$, and hence $r_1(Y' \cup X_{> i}) = |X|$.
This implies that $Y' \cup X_{> i}$ contains exactly one circuit of $\bM_1$, which must intersect $X_{> i}$ since $Y' \in \cI_1$.
Hence, there exists an element $x \in X_{> i}$ such that $Y' \cup (X_{> i} - x) \in \cI_1$.
The same holds for $\bM_2$.
Thus we obtain a common independent set $Y \subseteq Y' \cup X_{> i}$ with $|X_{> i} \setminus Y_{> i}| \le 2$, which is eligible.

\section{Concluding Remarks}
In this paper, we have analyzed how good a lex-maximal solution is in the weighted matching and matroid intersection problems based on how dispersed the distinct weight values are.
It is well-known that, subject to a single matroid, a lex-maximal solution is always optimal, which can be found by a greedy algorithm.
For more general independence systems, Jenkyns \cite{jenkyns1976efficacy} and Korte and Hausmann \cite{korte1978analysis} independently analyzed the worst-case approximation ratio of a greedy algorithm.
In particular, the worst ratio is $2$ in the weighted matching and matroid intersection problems, while a lex-maximal solution (which is a possible output of a greedy algorithm) always achieves the maximum weight if the distinct weight values are sufficiently dispersed.
From this perspective, we have filled the gap between these two situations.

A natural question is as follows: what about a further (common) generalization, e.g., the weighted matroid parity problem?
We just remark that it seems difficult to extend our algorithmic proof straightforwardly bacause no counterpart of augmentation from any $\ell$-extreme solution to some $(\ell+1)$-extreme solution along a path (Lemma~\ref{lem:extreme}) is known.
A minimal counterexample proof (like the matching case) also seems nontrivial; in particular, we have found no counterpart of Claim~\ref{cl:same_matching}.

\section*{Acknowledgments}
We are grateful to Naoyuki Kamiyama for providing comments on a literature on rank-maximal matchings.
We appreciate the anonymous reviewer's careful reading and helpful comments.
In particular, the alternative proof of Lemma~\ref{lem:matching} based on the augmentation property of rank-maximal matchings came from their comment.

Krist\'of B\'erczi was supported by the Lend\"ulet Programme of the Hungarian Academy of Sciences -- grant number LP2021-1/2021 and by the Hungarian National Research, Development and Innovation Office -- NKFIH, grant number FK128673.
Tam\'as Kir\'aly was supported by the Hungarian National Research, Development and Innovation Office -- NKFIH, grant number K120254.
Yutaro Yamaguchi was supported by JSPS KAKENHI Grant Numbers JP20K19743 and JP20H00605 and by Overseas Research Program in Graduate School of Information Science and Technology, Osaka University.
Yu Yokoi was supported by JSPS KAKENHI Grant Number JP18K18004 and JST, PRESTO Grant Number JPMJPR212B.

``Application Domain Specific Highly Reliable IT Solutions'' project  has been implemented with the support provided from the National Research, Development and Innovation Fund of Hungary, financed under the Thematic Excellence Programme TKP2020-NKA-06 (National Challenges Subprogramme) funding scheme.

\bibliographystyle{abbrv}
\bibliography{Doubling}

\end{document}